\documentclass[a4,10pt,reqno]{amsart}
\usepackage{amsfonts,amssymb}
\theoremstyle{plain}

\newtheorem{theorem} {Theorem}[section]

\newtheorem{definition}{Definition}[section]
\newtheorem{lemma}[theorem]{Lemma}

\newtheorem{rem}{Remark}[section]
\newcommand{\mat}{\mathbb}
\newcommand{\cal}{\mathcal}

\newcommand{\E}{\mathbb{E}}

\newcommand{\LL}{\mathbb{L}}

\newcommand{\R}{\mathbb{R}}

\newcommand{\K}{\mathbb{K}}

\newcommand{\PP}{\mathbb{P}}

\begin{document}
	\title[S-asymptotically $\omega$-periodic solutions in the $p$-th mean]{
On the $p$-th mean $S$-asymptotically omega periodic solution for
some Stochastic Evolution Equation driven by $\cal{Q}$-Brownian motion}

\maketitle{}

\centerline{Solym Mawaki MANOU-ABI* $^{\rm 1,2}$, William DIMBOUR $^{\rm 3}$}

\vspace{12pt}

\centerline{$^{\rm 1}$ CUFR  de Mayotte}
\centerline{D\'epartement Sciences Technologies, 97660 Dembeni }
\centerline{solym.manou-abi@univ-mayotte.fr }

\vspace{12pt}

\centerline{$^{\rm 2}$ Institut Montpelli\'erain Alexander Grothendieck}
\centerline{UMR CNRS 5149, Universit\'e de Montpellier 2}
\centerline{solym-mawaki.manou-abi@umontpellier.fr}

\vspace{12pt}

\centerline{$^{\rm 3}$ UMR Espace-Dev, Universit\'e de Guyane}
\centerline{Campus de Troubiran 97300}
\centerline{ Cayenne  Guyane (FWI)}
\centerline{william.dimbour@espe-guyane.fr}

\renewcommand{\thefootnote}{}
\footnote{*Corresponding author}
\renewcommand{\thefootnote}{\arabic{footnote}}
\setcounter{footnote}{0}
\date{}
\begin{abstract}
 In this paper, we make a slight contribution about the existence (uniqueness) and asymptotic stability  of the  $p$-th mean $S$-asymptotically $\omega$-periodic solutions for some nonautonomous Stochastic Evolution Equations driven by a $\mathcal{Q}$-Brownian motion. This is done using the Banach fixed point Theorem and a Gronwall inequality.
\end{abstract}
\vspace{12pt}



\section{Introduction}
Let $(\Omega,\mathcal{F},\mat{P})$ be a complete probability space and $(\mat{H},||.||)$ a real separable Hilbert space. We are concerned in this paper with the existence and
asymptotic stability of $p$-th mean $S$-asymptotically $\omega$-periodic solution of the following stochastic evolution equation 

\begin{equation}
\label{eqn: eq3050}
\left\{
\begin{array}{l}
dX(t) =  A(t) X(t)dt +f(t,X(t))dt  + g(t,X(t))dW(t), \quad \quad t \geq 0 \\
X(0) =  c_{0},
\end{array}
\right.
\end{equation}
where $(A(t))_{t\ge 0}$ is a familly of densely defined closed linears operators which generates an exponentially stable $\omega$-periodic two-parameter evolutionnary familly. The functions $ f: \R_{+} \times L^{p}(\Omega, \mat{H}) \rightarrow  \LL^{p}(\Omega, \mat{H})$, $  g: \R_{+} \times \LL^{p}(\Omega, \mat{H}) \rightarrow  \LL^{p}(\Omega, L_{2}^{0})$ 
are continuous satisfying some additional conditions and $(W(t))_{t\geq 0}$ is a $\cal{Q}$-Brownian  motion. The spaces  $\LL^{p}(\Omega, \mat{H})$, $L_{2}^{0}$  and the $\cal{Q}$-Brownian  motion are defined in the next section. \\
 
The concept of periodicity is important in probability especially for investigations
on stochastic processes. The interest in such a notion lies in its significance and applications arising in engineering, statistics, etc.  In recent years, there has been an increasing interest in periodic solutions (pseudo-almost periodic, almost periodic, almost automorphic, asymptotically almost periodic, etc) for stochastic evolution equations. For instance amoung others, let us mentioned the existence, uniqueness and asymptotic stability results of almost periodic solutions,  almost automorphic solutions, pseudo almost periodic solutions  studied by many authors, see, e.g. (\cite{bezandry,bez, beza,bezan,bezand,cao,chang,diop,xiliang,sun,zhang}).
The concept of $S$-asymptotically $\omega$-periodic stochastic processes, which is the
central question to be treated in this paper, was first introduced in the literature 
 by Henriquez, Pierri et al in (\cite{henriquez1,henriquez2}). This notion has been devellopped by many authors.\\
In the literature, there has been a significant attention devoted this  concept in the deterministic case; we refer the reader to (\cite{blot,cuevas2009,cuevas2010,cuevas2013,dimbour1,dimbour2}) and the references therein. However, in the random case, there are  few works related to 
the notion of $S$-asymptotically $\omega$-periodicity with regard to the existence, uniqueness and  asymptotic stability for stochastic processes. To our knowledge, the first work dedicated to $S$-asymptotically $\omega$-periodicity for stochastic processes  is due to S. Zhao and M. Song (\cite{zhao,song}) where they show existence of square-mean  $S$-asymptotically $\omega$-periodic solutions for a class of stochastic fractional functional differential equations and for a certain class of stochastic fractional evolution equation driven by Levy noise. But until now and to the best our knowledge,
there is no investigations for the  existence (uniqueness), asymptotic stability of $p$-th mean $S$-asymptotically $\omega$-periodic solutions when $p>2$.\\

This paper is organized as follows. Section 2 deals with some preliminaries intended to clarify the presentation of concepts and norms used latter. We also give a composition result, see Theorem \ref{composition}. In section 3 we present theoretical results on the existence and uniqueness of $S$-asymptotically $\omega$-periodic solution of equation (\ref{eqn: eq3050}), see Theorem \ref{existence}. We also present results on asymptotic stablity of the unique $S$-asymptotically $\omega$-periodic solution of equation(\ref{eqn: eq3050}), see Theorem \ref{stability}.

\section{Preliminaries}
This section is concerned with some notations, definitions, lemmas and preliminary facts which are used
in what follows. 
\subsection{p-th mean S asymptotically omega periodic process}
Assume that the probability space $(\Omega, \mathcal{F},\mathbb{P})$ is equipped with some filtration $(\mathcal{F}_{t})_{t\geq 0}$ satisfying the usual conditions. Let $p\geq 2$. Denote by  $\LL^{p}(\Omega,\mat{H})$  the collection of all strongly measurable $p$-th integrable $\mat{H}$-valued random variables such that
$$ \E ||X|| = \int_{\Omega} ||X(\omega)||^{p}d\PP(\omega) < \infty.  $$



\begin{definition}
A stochastic process $X : \R_+ \rightarrow \LL^{p}(\Omega, \mat{H})$ is said to be continuous whenever 
$$  \lim_{t\rightarrow s} \E || X(t)-X(s)||^{p} = 0.       $$
\end{definition} 

\begin{definition}
A stochastic process $X : \R_+ \rightarrow \LL^{p}(\Omega, \mat{H})$ is said to be bounded if 
there exists a constant $C > 0$ such that $$ \E ||X(t)||^{p} \leq C \quad \forall t \geq 0 $$
\end{definition}

\begin{definition}
A continous  and  bounded stochastic process $X : \R_{+} \rightarrow \LL^{p}(\Omega, \mat{H})$ is said to be $p$-mean $S$-asymptotically $\omega$ periodic if there exists $\omega >0$ such that 
 $$\lim_{t \rightarrow +\infty } \E || X(t+\omega)-X(t)||^{p} = 0, \quad \forall t \geq 0.$$
\end{definition}

The collection of $p$-mean $S$-asymptotically $\omega$-periodic stochastic process with values in $\mat{H}$ is then denoted by $SAP_{\omega}\big(\LL^{p}(\Omega, \mat{H})\big) $.\\

A continuous  bounded stochastic process $X$, which is $2$-mean  $S$-asymptotically 
$\omega$-periodic is also called square-mean $S$-asymptotically $\omega$-periodic.

\begin{rem}
Since any $p$-mean $S$-asymptotically $\omega$-periodic process $X$ is $\LL^{p}(\Omega, \mat{H})$  bounded and continuous, the space $SAP_{\omega}\big(\LL^{p}(\Omega, \mat{H})\big) $ is a Banach space  equipped with the sup norm : 
 $$ ||X||_{\infty} = \sup_{t\geq 0} \Big(\E || X(t) ||^{p}  \Big)^{1/p}.    $$
\end{rem}

 \begin{definition}
A function $F: \R_{+} \times \LL^{p}(\Omega, \mat{H}) \rightarrow  \LL^{p}(\Omega, \mat{H})$ which is jointly continuous, is said to be $p$-mean $S$-asymptotically $\omega$ periodic in $t \in \R_{+}$ uniformly in $X \in K$ where $K \subseteq \LL^{p}(\Omega, \mat{K})$ is bounded if  for any $\epsilon >0$ there exists $L_{\epsilon} >0$ such that 
$$  \E || F(t+\omega,X)-F(t,X)||^{p}\big] \leq \epsilon$$
for all $t \geq L_{\epsilon}$ and all process  $X : \R_{+} \rightarrow K $
\end{definition} 

\begin{definition}
A function $F: \R_{+} \times \LL^{p}(\Omega, \mat{H}) \rightarrow  \LL^{p}(\Omega, \mat{H})$ which is jointly continuous, is said to be $p$-mean asymptotically uniformly continuous on bounded sets $K' \subseteq \LL^{p}(\Omega, \mat{H})$, if for all $\epsilon > 0$ there exists $\delta_{\epsilon} >0$   such that 
$$ \E || F(t,X)-F(t,Y)||^{p} \leq \epsilon$$
for all $t \geq \delta_{\epsilon} $ and every $X, Y \in K'$ with $ \E||X-Y||^{p} \leq \delta_{\epsilon}$.
\end{definition}

\begin{theorem}
\label{composition}
Let $F: \R_{+} \times \LL^{p}(\Omega, \mat{H}) \rightarrow  \LL^{p}(\Omega, \mat{H})$ be a $p$-mean $S$-asymptotically $\omega$ periodic in $t \in \R_{+}$ uniformly in $X \in K$ where $K \subseteq \LL^{p}(\Omega, \mat{H})$ is  bounded and  $p$-mean asymptotically uniformly continuous on bounded sets. Assume that $X : \R_{+} \rightarrow \LL^{p}(\Omega, \mat{H})$ is a $p$-mean $S$ asymptotically $\omega$-periodic process. Then the stochastic process $ (F(t,X(t)))_{t\geq 0}$ is  $p$-mean $S$-asymptotically $\omega$ periodic.
\end{theorem}

\begin{proof}
Since $X : \R_{+} \rightarrow \LL^{p}(\Omega, \mat{H})$ is a $p$-mean $S$-asymptotically $\omega$-periodic process, for all $\epsilon >0 $, there exists $T_{\epsilon}>0$ such that for all $t \geq T_{\epsilon}$:
\begin{equation}
\label{eq:etoile1}
    \E || X(t+\omega)-X(t)||^{p} \leq \epsilon.
    \end{equation}
In addition $X$ is bounded that is $$\sup_{t\geq 0} \E || X(t) ||^{p} < \infty.$$
 Let $K \subset \LL^{p}(\Omega, \mat{H}) $ be a bounded set such that $X(t) \in K$ for all $t\geq 0$.\\
We  have :  
\begin{align*}
\E \left|\left| F(t+\omega,X(t+\omega))- F(t, X(t)) \right|\right|^{p}
 & \leq 2^{p-1} \E \left|\left| F(t+\omega,X(t+\omega))- F(t+\omega, X(t)) \right|\right|^{p} \\
 & \quad \quad  + 2^{p-1} \E \left|\left| F(t+\omega,X(t))- F(t, X(t)) \right|\right|^{p}
\end{align*}
Taking into account (\ref{eq:etoile1}) and using the fact that $F$ is $p$-mean asymptotically uniformly
continuous on bounded sets, there exists $\delta_{\epsilon}=\epsilon$ and $L_{\epsilon}= T_{\epsilon}$ such that for all $t\geq T_{\epsilon}$ : 
\begin{equation}
\label{eq:etoile2}
 \E \left|\left| F(t+\omega,X(t+\omega))- F(t+\omega, X(t)) \right|\right|^{p} \leq \frac{\epsilon}{2^{p}}.
\end{equation}
Similarly, using the $p$-mean $S$-asymptotically $\omega$ periodicity in $t \geq 0$ uniformly on bounded sets of $F$ it follows that for all $t\geq T_{\epsilon}$ : 
\begin{equation}
\label{eq:etoile3}
 \E \left|\left| F(t+\omega,X(t))- F(t, X(t)) \right|\right|^{p} \leq \frac{\epsilon}{2^{p}}.
 \end{equation} 
Bringing together the inequalities (\ref{eq:etoile2}) and (\ref{eq:etoile3}), we thus obtain that  for all $ t \geq T_{\epsilon} >0$ 
$$  \E \left|\left| F(t+\omega,X(t+\omega))- F(t, X(t)) \right|\right|^{p} \leq \epsilon        $$
so that the stochastic process $t\rightarrow F(t,X(t))$ is $p$-mean $S$-asymptotically $\omega$-
periodic.
\end{proof}

\begin{lemma}
\label{composition}
Assume that $F: \R_{+} \times \LL^{p}(\Omega, \mat{H}) \rightarrow  \LL^{p}(\Omega, \mat{H})$ is  $p$-mean uniformly $S$-asymptotically $\omega$-periodic in $t \in \R_{+}$ uniformly on bounded sets and 
satisfies the Lipschitz condition, that is, there exists  constant  $L(F)\, > 0$  such that
$$   \E ||F(t,X)-F(t,Y)||^{p} \leq L(F)\, \E||X-Y||^{p} \quad \forall t \geq 0, \, \forall X,Y \in \LL^{p}(\Omega, \K).       $$ 
Let $X$ be an $p$-mean $S$ asymptotically $\omega$-periodic proces, then the process $ (F(t,X(t)))_{t\geq 0}$ is  $p$-mean $S$-asymptotically $\omega$-periodic.
\end{lemma}
For the proof, the reader can refer to \cite{song} whenever $p=2$.  The case $p>2$ is similar.\\

Now let us recall the notion of evolutionary family of operators.
\begin{definition}
A two-parameter family of bounded linear operators $ \{ U(t,s) : t \geq s \; \textrm{with}\; t,s \geq 0 \} $
from $\LL^{p}(\Omega, \mat{H}))$ into itself  associate with $A(t)$ is called an evolutionary family of operators whenever the following conditions hold:
\begin{itemize}
\item[(a)] $$  U(t,s) U(s,r) = U(t,r)\quad \textrm{ for }\, \textrm{every} \; r \leq s \leq t ;  $$
\item[(b)] $$  U(t,t) = I , \; \textrm{where}\; I\; \textrm{is the identity operator}\, ; $$
\item[(c)] For all $X\in \LL^{p}(\Omega, \mat{H}))$, the function $(t,s) \rightarrow U(t,s)X$ is continuous for $s<t$ ;
\item[(d)] The function $t \rightarrow U(t,s)$ is differentiable and 
$$  \frac{\partial}{\partial t} (U(t,s)) = A(t) U(t,s)\quad \textrm{ for }\, \textrm{every}\;\; r \leq s \leq t ;  $$
\end{itemize}
\end{definition}
For additional details on evolution families, we refer the reader to the book by Lunardi \cite{lunardi}. 


\subsection{ $\cal{Q}$-Brownian motion and Stochastic integrals}
Let $(B_{n}(t))_{n\geq 1}$, $ t\geq 0$ be a sequence of real valued standard Brownian motion mutulally independent on the filtered space $ (\Omega,\cal{F},\PP,\cal{F}_{t})$. Set 
$$ W(t) = \sum_{n\geq 1} \sqrt{\lambda_{n}} B_{n}(t)\,e_{n},\quad t \geq 0,    $$
where $\lambda_{n} \geq 0$, $n\geq 1$, are non negative real numbers and $(e_{n})_{n\geq 1}$ the complete 
orthonormal basis in  the Hilbert space $(\mathbb{H}, ||.||)$.\\
Let $\mathcal{Q}$ be a symmetric nonnegative operator  with finite trace defined by $$ \mathcal{Q}e_{n}=\lambda_{n}e_{n}\quad \quad \textrm{such}\; \textrm{that}\quad Tr(\mathcal{Q}) =\sum_{n\geq 1} \lambda_{n} < \infty.$$
It is well known that $\mat{E} [ W_{t} ] = 0$ and for all $t\geq s \geq 0$, the distribution of $W(t)-W(s)$ is a Gaussian distribution ($\cal{N}(0,(t-s)\mathcal{Q})$).
The above-mentioned $\mat{H}$-valued stochastic process $(W(t))_{t\geq 0}$ is called an $\mathcal{Q}$-Brownian motion.\\ 
Let $(\mat{K},||.||_{K})$ be a real separable Hilbert space. \\
Let also $\mathcal{L}(\mat{K},\mat{H})$ be the space of all bounded linear operators from $\mat{K}$ into $\mat{H}$. If $\mat{K}=\mat{H}$, we denote it by $\mathcal{L}(\mat{H})$.\\

Set $\mat{H}_{0}= \mathcal{Q}^{1/2}\mat{H}$. The space $\mat{H}_{0}$ is a Hilbert space equipped with the norm $ ||u||_{\mat{H}_{0}}= ||\mathcal{Q}^{1/2}u||$.
Define 
$$ L_{2}^{0} = \{ \Phi \in  \mathcal{L}(\mat{H}_{0},\mat{H}): Tr\big[ (\Phi \cal{Q} \Phi^{*})  \big] < \infty  \}    $$ 
the space of all Hilbert-Schmidt operators from $\mat{H}_{0}$ to $\mat{H}$ equipped with the norm 
$$  ||\Phi||_{L_{2}^{0}} = Tr\big[ (\Phi \cal{Q} \Phi^{*})  \big]=\E||\Phi\mathcal{Q}^{1/2 } ||^{2}$$

In the sequel, to prove Lemma \ref{wedge2} and Theorem \ref{existence} we need the following Lemma that is a particular case of Lemma 2.2 in \cite{seidlerbis} (see also  \cite{prato,seidler}).\\

Assume $T>0$. 
 \begin{lemma}
 \label{itotype}
Let $G: [0,T]\rightarrow \mathcal{L}(\LL^{p}(\Omega,\mat{H}))$ be an 
  $\cal{F}_{t}$-adapted measurable stochastic process satisfying 
$  \int_{0}^{T} \E||G(t)||^{2}dt < \infty \quad \textrm{almost surely}.  $
Then, 
 \begin{itemize}
\item[(i)] the stochastic integral $ \int_{0}^{t} G(s)dW(s)$ is a continuous, square integrable martingale with values in $(\mat{H},||.||)$  such that
  $$ \E \left|\left| \int_{0}^{t} G(s)dW(s)  \right|\right|^{2} \leq \E \int_{0}^{t}  || G(s)||^{2} ds.   $$
 \item[(ii)] There exists some constant $C_{p} >0$ such that the following particular case of  Burkholder-Davis-Gundy inequality holds :
  $$ \E \sup_{0\leq t \leq T}\left|\left| \int_{0}^{t} G(s)dW(s)  \right|\right|^{p} \leq  C_{p}  \E  \left(\int_{0}^{T}|| G(s)||^{2} ds \right)^{p/2}.   $$   
 \end{itemize}
 \end{lemma}


In the sequel, we'll frequently make use of the following inequalities :
$$  |a+b|^{p} \leq 2^{p-1} (|a|^{p} + |b|^{p})\quad \; \textrm{for all} \; p\geq 1\; \textrm{and any real numbers} \; a, \, b.  $$

$$  \int_{t_{0}}^{t} e^{-2a(t-s)}ds \leq \int_{t_{0}}^{t} e^{-a(t-s)}ds \leq \frac{1}{a}\quad \forall t \geq t_{0},\; \; \textrm{where} \; a >0. $$
\section{Main results}
In this section, we investigate the existence and the asymptotically stability of the $p$-th mean  $S$-asymptotically $\omega$-periodic solution to the already  defined stochastic differential equation :

\begin{equation*}
dX(t)=A(t) X(t)dt +f(t,X(t))dt + g(t,X(t))dW(t),  \quad \quad 
X(0)= c_{0}
\end{equation*}
where $A(t), t\geq 0$ is a family of densely defined closed linear operators and 
$$   f: \R_{+} \times \LL^{p}(\Omega, \mat{H}) \rightarrow  \LL^{p}(\Omega, \mat{H}),    $$

$$  g: \R_{+} \times \LL^{p}(\Omega, \mat{H}) \rightarrow  \LL^{p}(\Omega, L_{2}^{0})  $$
are jointly continuous satisfying some additional conditions and $(W(t))_{t\geq 0}$ is a $\cal{Q}$-Brownian motion with values in $\mat{H}$ and $\cal{F}_t$-adapted.

 Throughout the rest of this section, we require the following assumption on $U(t,s)$ :\\

{ \bf (H1)}: $A(t)$ generates an exponentially $\omega$-periodic stable evolutionnary process $(U(t,s))_{t\geq s}$ in $\LL^{p}(\Omega, \mat{H})$, that is, a two-parameter family of bounded linear operators with the following additional conditions :\\

\begin{itemize}
	\item[1.] $ U(t+\omega, s+\omega)= U(t,s)$ for all $t\geq s$ ($\omega$-periodicity).
	\item[2.] There exists $M>0$ and $a>0$ such that $||U(t,s)|| \leq Me^{-a(t-s)}$ for $t\geq s.$\\
\end{itemize}

Now, note  that if $A(t)$ generates an evolutionary family $(U(t,s))_{t\geq s}$ on $\LL^{p}(\Omega, \mat{H}))$ then the function $g$ defined by $g(s) = U(t,s)X(s)$ where $X$ is a solution of equation (\ref{eqn: eq3050}), satisfies the following relation
\begin{align*}
	dg(s)
	&= -A(s)U(t,s)X(s) + U(t,s)dX(s) \\
	&=  -A(s)U(t,s)X(s)+ A(s)U(t,s)X(s)ds+ U(t,s)f(s,X(s))ds + U(t,s)g(s,X(s))dW(s).
\end{align*}
Thus
\begin{equation}
\label{eq:aintegrer}
dg(s) = U(t,s)f(s,X(s))ds + U(t,s)g(s,X(s))dW(s).
\end{equation}

Integrating (\ref{eq:aintegrer}) on $[0,t]$ we obtain that
$$ X(t)- U(t,0)c_0 =  \int_{0}^{t} U(t,s)f(s,X(s))ds + \int_{0}^{t} U(t,s)g(s,X(s))dW(s) . $$

Therefore, we define
\begin{definition}
An $(\mathcal{F}_{t})$-adapted stochastic process  $ (X(t))_{t\geq 0 }$ is called a mild solution of (\ref{eqn: eq3050}) if it satisfies the following stochastic integral equation : 
	$$ X(t) = U(t,0)c_{0} + \int_{0}^{t} U(t,s)f(s,X(s))ds + \int_{0}^{t} U(t,s)g(s,X(s))dW(s).$$
\end{definition}

\subsection{The existence of $p$-th mean $S$-asymptotically $\omega$-periodic solution} 
We require the following additional assumptions:\\

{\bf (H.2)} The function $f: \R_{+} \times \LL^{p}(\Omega, \mat{H}) \rightarrow  \LL^{p}(\Omega, \mat{H})$  is  $p$-mean $S$-asymptotically $\omega$ periodic in $t \in \R_{+}$ uniformly in $X \in K$ where $K \subseteq \LL^{p}(\Omega, \mat{H})$ is a bounded set. Moreover the function $f$ satisfies the Lipschitz condition, that is, there exists  constant  $L(f)\, > 0$  such that
$$   \E ||f(t,X)-f(t,Y)||^{p} \leq L(f) \E||X-Y||^{p} \quad \forall t \geq 0, \, \forall X,Y \in \LL^{p}(\Omega, \mat{H}).       $$

{\bf (H.3)}  The function $g: \R_{+} \times \LL^{p}(\Omega, L_{2}^{0}) \rightarrow  \LL^{p}(\Omega, L_{2}^{0})$  is  $p$-mean $S$-asymptotically $\omega$ periodic in $t \in \R_{+}$ uniformly in $X \in K$ where $K \subseteq \LL^{p}(\Omega,L_{2}^{0})$ is a bounded set. Moreover the function $g$  satisfies the Lipschitz condition, that is, there exists  constant  $L(g) \, > 0$  such that
$$   \E ||g(t,X)-g(t,Y)||^{p}_{L_{2}^{0}} \leq L(g) \E||X-Y||^{p} \quad \forall t \geq 0, \, \forall X,Y \in \LL^{p}(\Omega,\mat{H}).       $$ 

\begin{lemma}
\label{wedge1}
We assume that hypothesis {\bf(H.1)} and {\bf(H.2)}  are  satisfied.  We define the nonlinear operator $\wedge_1$ by: for each $\phi \in SAP_{\omega}(\mathbb{L}^{p}(\Omega, \mat{H}))$
	$$(\wedge_1 \phi)(t)=\int_{0}^t U(t,s) f(s,\phi(s))ds.$$
	Then the operator $\wedge_1$ maps $SAP_{\omega}(\mathbb{L}^{p}(\Omega, \mat{H}))$ into itself.
\end{lemma}

\begin{proof}
We define $h(s)= f(s,\phi(s))$. Since the hypothesis  {\bf(H.2)} is satisfied, using Lemma \ref{composition}, we deduce that the function $h$ is $p$-mean S-asymptotically $\omega$-periodic.\\
Define $F(t)= \int_{0}^t U(t,s) h(s)ds $. It is easy to check that $F$ is bounded and continuous.
Now we have :   
\begin{align*}
F(t+\omega)-F(t)
&= \int_{0}^{\omega}U(t+\omega,s)h(s)ds +
\int_{0}^{t} U(t,s)\big(h(s+\omega)-h(s)\big)ds \\
& = U(t+\omega, \omega) \int_{0}^{\omega}U(\omega,s)h(s)ds+
\int_{0}^{t} U(t,s)\big(h(s+\omega)-h(s)\big)ds
\end{align*}


\begin{align*}
\E || F(t+\omega)-F(t)||^{p}
 & \leq  2^{p-1} M^{2p}e^{-apt} \E \Big(\int_{0}^{\omega} e^{-a(\omega-s)} ||h(s)|| ds  \Big)^{p} \\
 & +  2^{p-1}M^{p} \E \Big(\int_{0}^{t}  e^{-a(t-s)} \left|\left|h(s+\omega)-h(s) \right|\right|ds\Big)^{p}
\end{align*} 
 
Let $p$ and $q$ be conjugate exponents. Using H\"older  inequality, we obtain that 

\begin{align*}
\mathbb{E} || F(t+\omega)-F(t)||^{p} 
& \leq   2^{p-1} M^{2p}e^{-apt}  \Big(\int_{0}^{\omega} e^{-aq(\omega-s)} ds \Big)^{p/q} \int_{0}^{\omega} \mathbb{E}||h(s)||^{p} ds  \\
&  +  2^{p-1}M^{p} \mathbb{E} \Big(\int_{0}^{t}  e^{-a(t-s)} \left|\left|h(s+\omega)-h(s) \right|\right|ds\Big)^{p} \\
& =  I(t) +  J(t)
\end{align*} 
where 
$$ I(t) =  2^{p-1} M^{2p}e^{-apt}  \Big(\int_{0}^{\omega} e^{-aq(\omega-s)}ds \Big)^{p/q} \int_{0}^{\omega} \mathbb{E}||h(s)||^{p} ds    $$
$$ J(t) =  2^{p-1}M^{p}\mathbb{E} \Big(\int_{0}^{t}  e^{-a(t-s)} \left|\left|h(s+\omega)-h(s) \right|\right|ds\Big)^{p}$$
$$ \quad \quad \quad \quad \quad    \quad \quad \quad  =  2^{p-1}M^{p} \mathbb{E} \Big(\int_{0}^{t}  e^{-\frac{a}{q}(t-s)}\times  e^{-\frac{a}{q}(t-s)} \left|\left|h(s+\omega)-h(s) \right|\right|ds\Big)^{p}.  $$

It is obvious that 
$$ \displaystyle{\lim_{t\rightarrow +\infty}}I(t) =0.     $$

Using H\"older  inequality, we obtain that 

 \begin{align*}
J(t)
 & \leq 2^{p-1}M^{p}\Big(\int_{0}^{t} e^{-a(t-s)}ds  \Big)^{p/q} \int_{0}^{t}
   e^{-a(t-s)} \mathbb{E} \left|\left|h(s+\omega)-h(s) \right|\right|^{p}ds\\
   & \leq 2^{p-1}M^{p} \big( \frac{1}{a} \big)^{p/q}\int_{0}^{t}
   e^{-a(t-s)} \mathbb{E} \left|\left|h(s+\omega)-h(s) \right|\right|^{p}ds\\
\end{align*}

Let $\epsilon >0$. Since $ \displaystyle{\lim_{u\rightarrow +\infty}} \mathbb{E} ||h(u+\omega)-h(u)||^{p} = 0 $ : 
\begin{equation}
\label{eq:majoration}
 \exists T_{\epsilon} >0,\, u>T_{\epsilon} \Rightarrow 
\mathbb{E} ||h(u+\omega)-h(u)||^{p} \leq \frac{\epsilon a^{p}}{2^{p-1}M^{p}}. 
\end{equation}

We have 
 \begin{align*}
J(t) 
& \leq 2^{p-1}M^{p} \big( \frac{1}{a} \big)^{p/q} \int_{0}^{T_{\epsilon}}  e^{-a(t-s)} \mathbb{E} \left|\left|h(s+\omega)-h(s) \right|\right|^{p} ds\\
& +  2^{p-1}M^{p}\big( \frac{1}{a} \big)^{p/q} \int_{T_{\epsilon}}^{t}  e^{-a(t-s)} \mathbb{E}\left|\left|h(s+\omega)-h(s) \right|\right|^{p}ds \\
&= J_{1}(t) + J_{2}(t),
\end{align*} 
where 
$$  J_{1}(t) = 2^{p-1}M^{p} \big( \frac{1}{a} \big)^{p/q} \int_{0}^{T_{\epsilon}}  e^{-a(t-s)} \mathbb{E} \left|\left|h(s+\omega)-h(s) \right|\right|^{p} ds$$

$$ J_{2}(t) =  2^{p-1}M^{p} \big( \frac{1}{a} \big)^{p/q} \int_{T_{\epsilon}}^{t}  e^{-a(t-s)} \mathbb{E} \left|\left|h(s+\omega)-h(s) \right|\right|^{p}ds$$

\underline{Estimation of} $J_{1}(t)$.\\

\begin{eqnarray*}
J_{1}(t) & \leq & 2^{p-1}M^{p}\big( \frac{1}{a} \big)^{p/q} \int_{0}^{T_{\epsilon}}  e^{-a(t-s)} \mathbb{E} \left|\left|h(s+\omega)-h(s) \right|\right|^{p} ds \\
 & \leq &  2^{p-1}M^{p}\big( \frac{1}{a} \big)^{p/q} \,2^{p}\, \sup_{t\geq 0} \E ||h(t)||^{p} e^{-at}\int_{0}^{T_{\epsilon}}  e^{aqs} ds .  
\end{eqnarray*} 
It is clear that $   \displaystyle{\lim_{t\rightarrow +\infty}} J_{1}(t) = 0.$\\

\underline{Estimation of} $J_{2}(t)$.\\

Unsing the Inequality in (\ref{eq:majoration}) we have 
\begin{eqnarray*}
J_{2}(t) & = & 2^{p-1}M^{p} \Big( \frac{1}{a} \Big)^{p/q} \int_{T_{\epsilon}}^{t}  e^{-a(t-s)} \mathbb{E} \left|\left|h(s+\omega)-h(s) \right|\right|^{p}ds\\
 & \leq  &  2^{p-1}M^{p} \Big( \frac{1}{a} \Big)^{p/q} \Big(\frac{1}{a}\Big)  \frac{\epsilon a^{p}}{2^{p-1}M^{p}} \\
& = &  2^{p-1}M^{p} a^{-p}\frac{\epsilon a^{p}}{2^{p-1}M^{p}}  \\
& \leq & \epsilon.
\end{eqnarray*} 

\end{proof}
\begin{lemma}
\label{wedge2}
We assume that hypothesis {\bf(H.1)} and  {\bf(H.3)}   are  satisfied.  We define the nonlinear operator $\wedge_2$ by: for each
 $\phi \in SAP_{\omega}(\mathbb{L}^{p}(\Omega, L_{2}^{0}))$
	$$(\wedge_2 \phi)(t)=\int_{0}^t U(t,s) g(s,\phi(s))dW(s).$$
Then the operator $\wedge_2$ maps $SAP_{\omega}(\mathbb{L}^{p}(\Omega, L_{2}^{0}))$ into itself.
\end{lemma}
\begin{proof}
We define $h(s)= g(s,\phi(s))$. Since the hypothesis  {\bf(H.3)} is satisfied, using Lemma \ref{composition}, we deduce that the function $h$ is $p$- mean S asymptotically $\omega$ periodic.\\
Define $F(t)= \int_{0}^t U(t,s) h(s)ds $. It is easy to check that $F$ is bounded and continuous.
We have :

$$F(t+\omega)-F(t)=\int_{0}^{\omega}U(t+\omega,s)h(s)dW(s) +
\int_{0}^{t} U(t,s)\big(h(s+\omega)-h(s)\big)dW(s) ,$$
$$= U(t+\omega, \omega) \int_{0}^{\omega}U(\omega,s)h(s)dW(s)+
\int_{0}^{t} U(t,s)\big(h(s+\omega)-h(s)\big)dW(s) ,$$

\begin{align*}
\E|| F(t+\omega)-F(t)||^{p} 
& \leq    2^{p-1}M^{p}e^{-apt} \E \left|\left|\int_{0}^{\omega}  U(\omega,s)h(s) dW(s)\right|\right|^{p}\\
& +  2^{p-1}\E  \left|\left| \int_0^t U(t,s)\big(h(s+\omega)-h(s)\big)dW(s) \right|\right|^{p} \\
& : =  I(t) + J(t)
\end{align*} 
where 

$$ I(t) = 2^{p-1} M^{p}e^{-pat}  \E \left|\left|\int_{0}^{\omega}  U(\omega,s)h(s) dW(s)\right|\right|^{p}  $$
$$ J(t)=   2^{p-1} \E  \left|\left|  \int_0^t U(t,s)\big(h(s+\omega)-h(s)\big)dW(s)  \right|\right|^{p} $$



  


It is clear that  $$ \displaystyle{\lim_{t\rightarrow +\infty}} \mathbb{E}I(t) =0.$$

 Let $\epsilon >0$. Since $ \displaystyle{\lim_{t\rightarrow +\infty}} \mathbb{E} ||h(t+\omega)-h(t)||^{p} = 0  $ 
\begin{equation}
\label{eq:majorationbis}
\exists T_{\epsilon} >0,\, t>T_{\epsilon} \Rightarrow \E  ||h(s+\omega)-h(s)||^{p}_{L^{2}_{0}} \leq  \frac{ \epsilon\, (2a)^{p/2}}{4^{p-1}M^{p}C_{p}},  
   \end{equation}
   where the constant $C_p$ will be precised in the next lines.\\
We have 
 \begin{align*}
\E J(t)
 & =     2^{p-1} \E \left|\left|  \int_0^t U(t,s)\big(h(s+\omega)-h(s)\big)dW(s)  \right|\right|^{p}  \\
&\leq  4^{p-1}  \E \left|\left|  \int_{0}^{T_{\epsilon}} U(t,s)\big(h(s+\omega)-h(s)\big)dW(s)  \right|\right|^{p} \\
& +   4^{p-1} \E \left|\left|  \int_{T_{\epsilon}}^{t} U(t,s)\big(h(s+\omega)-h(s)\big)dW(s)  \right|\right|^{p}\\
&: = \E J_{1}(t) + \E J_{2}(t),
\end{align*} 
where 
$$  J_{1}(t) =  4^{p-1} \left|\left|  \int_{0}^{T_{\epsilon}} U(t,s)\big(h(s+\omega)-h(s)\big)dW(s)  \right|\right|^{p} $$
$$ J_{2}(t) =   4^{p-1}  \left|\left|  \int_{T_{\epsilon}}^{t} U(t,s)\big(h(s+\omega)-h(s)\big)dW(s)  \right|\right|^{p}$$
\bigskip

Note that for all $t\geq 0$, $h(t+\omega)-h(t) \in \mathbb{L}^{p}(\Omega, L_{2}^{0}) \subseteq \mathbb{L}^{2}(\Omega, L_{2}^{0}) $ and  

\begin{align*}
\int_{0}^{t} \E||U(t,s)(h(s+\omega)-h(s))||^{2}ds &  \leq M^{2}\int_{0}^{t} e^{-2a(t-s)} \E||h(s+\omega)-h(s)||^{2}_{L_{2}^{0}}ds \\
& \leq 4M^{2}\sup_{t\geq 0} \E||h(t)||^{2}_{L_{2}^{0}} \int_{0}^{t}e^{-2a(t-s)} ds \\
& \leq  4M^{2}a^{-2} \sup_{t\geq 0} \E||h(t)||^{2}_{L_{2}^{0}}\\
& < \infty.
\end{align*}

\underline{Estimation of} $\E J_{1}(t)$.\\

Assume that $p>2$. Using H\"older inequality between
conjugate exponents  $\frac{p}{p-2}$ and $\frac{p}{2}$ together with Lemma \ref{itotype}, part (ii), there exists 
constant $C_p$ such that : 

\begin{align*}
\E J_{1}(t) 
& \leq  C_{p}4^{p-1}M^{p} \E \left[ \int_{0}^{T_{\epsilon}} e^{-2a(t-s)} ||h(s+\omega)-h(s)||^{2}_{L_{2}^{0}} ds  \right]^{p/2} \\
 & \leq   C_{p}4^{p-1}M^{p} \Big( \int_{0}^{T_{\epsilon}} e^{-\frac{2ap(t-s)}{p-2}}  ds \Big)^{\frac{p-2}{2}} \, \int_{0}^{T_{\epsilon}}  \E ||h(s+\omega)-h(s)||^{p}_{L_{2}^{0}} ds \\
  & \leq   C_{p}4^{p-1}M^{p}\, e^{-apt}\,\Big( \int_{0}^{T_{\epsilon}} e^{\frac{2aps}{p-2}} ds \Big)^{\frac{p-2}{2}}\,T_{\epsilon}2^{p}\,\sup_{s\geq 0} \E ||h(s)||^{p}_{L_{2}^{0}}.
\end{align*} 

Therefore 
 $$ \displaystyle{\lim_{t\rightarrow +\infty}} \mathbb{E} J_{1}(t) =0.$$
 
Assume that $p=2$. By Lemma \ref{itotype}, part (i) we get :
 \begin{align*}
\E J_{1}(t)
 & \leq  4M^{2} \E \left[  \int_{0}^{T_{\epsilon}} e^{-a(t-s)} ||h(s+\omega)-h(s)||_{L_{2}^{0}} dB(s)  \right]^{2}  \\
 & \leq   4M^{2} \E \left[ \int_{0}^{T_{\epsilon}} e^{-2a(t-s)} ||h(s+\omega)-h(s)||^{2}_{L_{2}^{0}} ds \right]\\
  & \leq  16M^{2} e^{-2at}\, \sup_{s\geq 0} \E|| h(s)||^{2}_{L_{2}^{0}} \int_{0}^{T_{\epsilon}} e^{2as}ds\\
    & \leq  16M^{2} e^{-2at} \, \sup_{s\geq 0} \E|| h(s)||^{2}_{L_{2}^{0}}\,\int_{0}^{T_{\epsilon}} e^{2as}ds
\end{align*} 
Thus 
 $$ \displaystyle{\lim_{t\rightarrow +\infty}} \mathbb{E} J_{1}(t) =0.$$

\underline{Estimation of} $\E J_{2}(t)$. \\

Assume that $p>2$.  Using again  Lemma \ref{itotype}, part (ii), H\"older inequality between  between
conjugate exponents  $\frac{p}{p-2}$ and $\frac{p}{2}$ and the inequality in (\ref{eq:majorationbis}) we have

 \begin{align*}
\E J_{2}(t) 
 & =   4^{p-1} \E \left|\left|  \int_{T_{\epsilon}}^{t} U(t,s)\big(h(s+\omega)-h(s)\big)dW(s)  \right|\right|^{p}  \\
& \leq   4^{p-1}M^{p} C_{p}  \E \left[  \int_{T_{\epsilon}}^{t} e^{-2a(t-s)} ||h(s+\omega)-h(s)||^{2}_{L_{2}^{0}} ds  \right]^{p/2} \\
& =  4^{p-1}M^{p} C_{p}  \E \left[  \int_{T_{\epsilon}}^{t} e^{-2a(t-s)\frac{p-2}{p}}\times e^{-2a(t-s)\frac{2}{p}}  ||h(s+\omega)-h(s)||^{2}_{L_{2}^{0}} ds  \right]^{p/2} \\
&\leq   C_{p}4^{p-1}M^{p} \Big(\int_{T_{\epsilon}}^{t}  e^{-2a(t-s)}  ds \Big)^{\frac{p-2}{2}} \,   \int_{T_{\epsilon}}^{t} e^{-2a(t-s)} \E  ||h(s+\omega)-h(s)||^{p}_{L_{2}^{0}} ds \\
&\leq  \frac{ C_{p}4^{p-1}M^{p}\epsilon (2a)^{p/2}}{C_{p}4^{p-1}M^{p}}\Big(\int_{T_{\epsilon}}^{t}  e^{-2a(t-s)}  ds \Big)^{\frac{p}{2}}\\
 & \leq \epsilon.
\end{align*} 

We conclude that $$  \displaystyle{\lim_{t\rightarrow +\infty}} \E J_{2}(t) = 0.$$

Assume that $p=2$.  By Lemma \ref{itotype}, part (i)  and Cauchy-Schwarz inequality we have 
\begin{align*}
\E J_{2}(t)  
& =  4 \E \left|\left|  \int_{T_{\epsilon}}^{t} U(t,s)\big(h(s+\omega)-h(s)\big)dW(s)  \right|\right|^{2}  \\
& \leq    4M^{2}  \E \left[  \int_{T_{\epsilon}}^{t} e^{-2a(t-s)} ||h(s+\omega)-h(s)||^{2}_{L_{2}^{0}} ds  \right] \\
& =  4M^{2}   \left[  \int_{T_{\epsilon}}^{t} e^{-a(t-s)}\times e^{-a(t-s)} \E ||h(s+\omega)-h(s)||^{2}_{L_{2}^{0}} ds  \right]  \\
&\leq  4M^{2} \Big( \int_{T_{\epsilon}}^{t} e^{-2a(t-s)}ds\Big)^{1/2} 
 \Big( \int_{T_{\epsilon}}^{t} e^{-2a(t-s)} \big(\E ||h(s+\omega)-h(s)||^{2}_{L_{2}^{0}}\big)^{2} ds
 \Big)^{1/2}
\end{align*} 
Note also that for $t \geq T_{\epsilon}$ : 
$$ \E  ||h(s+\omega)-h(s)||^{2}_{\mat{L}_{2}^{0}} \leq \frac{\epsilon\,a}{2M^{2}}     $$
so that 
\begin{align*}
\E J_{2}(t) 
&\leq  \frac{4M^{2}\epsilon a}{2M^{2}} \Big(\int_{T_{\epsilon}}^{t} e^{-2a(t-s)}ds\Big) \\
 & \leq  \epsilon.
\end{align*}
This implies that $$\displaystyle{\lim_{t\rightarrow +\infty}} \E J_{2}(t) = 0. $$
Finally, we conclude that 
$$   \displaystyle{\lim_{t\rightarrow +\infty}} \mathbb{E} || F(t+\omega)-F(t)||^{p} = 0           $$
\end{proof}
	
\begin{theorem}
\label{existence}
 We assume that hypothesis {\bf(H.1)}, {\bf(H.2)} and {\bf(H.3)}   are  satisfied and 
\begin{itemize}
\item[(i)] 
\begin{equation}
\label{eq:existence1}
\Theta   = 2^{p-1}M^{p}\big( L(f) a^{-p} + C_{p}L(g) a^{\frac{-p}{2}} \big) < 1  \quad \quad \textrm{if}\quad  p>2  
  \end{equation}
\item[(ii)] 
\begin{equation}
\label{eq:existence2}
  \Xi  = 2M^{2} \big( L(f) \frac{1}{a^{2}} + L(g) \frac{1}{a} \big) < 1 \quad \quad \textrm{if}\quad  p=2.     
    \end{equation} 
\end{itemize}
Then the stochastic evolution equation (\ref{eqn: eq3050}) has a unique $p$-mean $S$-asymptoticaly $\omega$-periodic solution.
\end{theorem}	
	
\begin{proof}
We define the nonlinear operator $\Gamma$  by the expression 

	$$  (\Gamma \Phi )(t) =  U(t,0) c_{0} +  \int_{0}^t U(t,s) f(s,\Phi(s))ds +  	\int_{0}^{t} U(t,s) g(s,\Phi(s))dW(s)             $$
Note that 	$$     (\Gamma \Phi )(t) =  U(t,0) c_{0} + (\wedge_{1}\Phi)(t) + (\wedge_{2}\Phi)(t)              $$
According to the hypothesis {\bf (H1)} we have :
\begin{eqnarray}
\E ||U(t+\omega,0)-U(t,0)||^{p}  & \leq & 2^{p-1} \big( \E ||U(t+\omega,0||^{p}+ \E ||U(t,0)||^{p}\big) \\
                        & \leq & 2^{p-1} M^{p}e^{-ap(t+\omega)} + M^{p}e^{-ap t} \\
                         & = & 2^{p-1} M^{p}e^{-apt}\big( e^{-ap\omega} + 1\big)
\end{eqnarray}
Therefore $$   \displaystyle{\lim_{t\rightarrow +\infty}} \E ||U(t+\omega,0)-U(t,0)||^{p}= 0.    $$
According to Lemma \ref{wedge1} and Lemma \ref{wedge2}, the operators $\wedge_1$ and $\wedge_2$ maps the space of
$p$-mean S-asymptotically $\omega$-periodic solutions into itself. Thus $\Gamma$ maps the space of $p$-mean S-asymptotically $\omega$ periodic solutions into itself. We have 
\begin{align*}
\E || \Gamma\Phi(t)-\Gamma \Psi(t)||^{p}  
&\leq  2^{p-1} \E \left(\int_{0}^{t} ||U(t,s)||\, ||f(s,\Phi(s))-f(s,\Psi(s))||ds \right)^{p}\\
& + 2^{p-1} \E \left( \int_{0}^{t} ||U(t,s)||\, ||g(s,\Phi(s))-g(s,\Psi(s))||dW(s) \right)^{p}\\
 & \leq   2^{p-1} M^{p} \E \left(\int_{0}^{t} e^{-a(t-s)}\, ||f(s,\Phi(s))-f(s,\Psi(s))||ds \right)^{p} \\
 &+ 2^{p-1}M^{p} \E \left( \int_{0}^{t} e^{-a(t-s)}\, ||g(s,\Phi(s))-g(s,\Psi(s))||dW(s) \right)^{p}
\end{align*}	

Case $p>2$ : By Lemma \ref{itotype}, part (ii) and H\"older inequality  we have

\begin{align*}
\E || \Gamma\Phi(t)-\Gamma \Psi(t)||^{p} 
 &\leq  2^{p-1}M^{p} L(f)\big(\int_{0}^{t} e^{-a(t-s)} ds \big)^{p-1}  \int_{0}^{t}e^{-a(t-s)} \E  ||\Phi(s)-\Psi(s)||^{p}ds \\
 & + 2^{p-1}M^{p}C_{p} \E \Big( \int_{0}^{t} e^{-2a(t-s)}\, ||g(s,\Phi(s))-g(s,\Psi(s))||_{L^{0}_{2}}^{2}ds \Big)^{p/2} \\
 &\leq  2^{p-1}M^{p} L(f) \sup_{s\geq 0} \E  ||\Phi(s)-\Psi(s)||^{p} \Big(\int_{0}^{t} e^{-a(t-s)} ds \Big)^{p} \\
 & +  2^{p-1}M^{p}C_{p} \E \Big( \int_{0}^{t} e^{-a(t-s)}\, ||g(s,\Phi(s))-g(s,\Psi(s))||_{L^{0}_{2}}^{2}ds \Big)^{p/2} \\
&\leq  2^{p-1}M^{p}a^{-p} L(f) ||\Phi - \Psi ||_{\infty}^{p}  +  \\
& 2^{p-1}M^{p}C_{p}\E \Big( \int_{0}^{t} e^{-\frac{ap(t-s)}{p-2}} e^{-\frac{a(t-s)}{p}} ||g(s,\Phi(s))-g(s,\Psi(s))||_{L^{0}_{2}}^{2}ds \Big)^{p/2} \\
&\leq  2^{p-1}M^{p}a^{-p} L(f) ||\Phi - \Psi ||_{\infty}^{p}\\
&  + 2^{p-1}M^{p}C_{p}\Big( \int_{0}^{t}e^{-a(t-s)}ds\Big)^{\frac{p-2}{2}} \times \\
&  \int_{0}^{t} e^{-a(t-s)} \E||g(s,\Phi(s))-g(s,\Psi(s))||_{L^{0}_{2}}^{p}ds\\
 & \leq  2^{p-1}M^{p}\big( L(f) a^{-p} +C_{p}L(g) a^{-\frac{p}{2}} \big)  ||\Phi - \Psi ||_{\infty}^{p} 
\end{align*}
This implies that 
\[ || \Gamma \Phi - \Gamma \Psi ||_{\infty}^{p} \leq  2^{p-1}M^{p}\big( L(f) a^{-p} +C_{p}L(g) a^{-\frac{p}{2}} \big)  ||\Phi - \Psi ||_{\infty}^{p}. \]
Consequently, if $\Theta < 1$, then $\Gamma$ is a contraction mapping. One completes the proof by the Banach fixed-point principle.\\

Case $p=2$ : using Cauchy-Schwarz inequality and Lemma \ref{itotype},  part (i), we obtain
\begin{align*}
\E || \Gamma\Phi(t)-\Gamma \Psi(t)||^{2}
  &\leq  2M^{2} \big( \int_{0}^{t} e^{-a(t-s)} ds \big)\, \E \int_{0}^{t}
e^{-a(t-s)} ||f(s,\Phi(s))-f(s,\Psi(s))||^{2}ds \\
& + 2M^{2} \E \int_{0}^{t} e^{-2a(t-s)}\, ||g(s,\Phi(s))-g(s,\Psi(s))||^{2}_{L_{2}^{0}}ds \\
 &\leq 2M^{2}L(f)\sup_{s\geq 0} \E || \Phi(s)-\Psi(s)||^{2}  \big( \int_{0}^{t} e^{-a(t-s)} ds \big)^{2}\\
 & + 2M^{2} L(g)\sup_{s\geq 0} \E || \Phi(s)-\Psi(s)||^{2}  \int_{0}^{t} e^{-2a(t-s)}ds\\
 & \leq  2M^{2} \big(L(f) \frac{1}{a^{2}} + L(g) \frac{1}{a}\big) \sup_{s\geq 0} \E || \Phi(s)-\Psi(s)||^{2}  
\end{align*}
	
This implies that 
\[ || \Gamma \Phi - \Gamma \Psi ||_{\infty}^{2} \leq 2M^{2} \big(L(f)\frac{1}{a^{2}} + L(g)\frac{1}{a}\big) 
|| \Phi - \Psi ||_{\infty}^{2}.  \]

Consequently, if $\Xi < 1$, then $\Gamma$ is a contraction mapping. One completes the proof by the Banach fixed-point principle. \\

\end{proof}

\subsection{Stability of $p$-mean S asymptotically $\omega$ periodic solution} 
In the previous section, for the non linear SDE, we obtain that it has a unique $p$-mean
 S-asymptotically $\omega$-periodic solution under some conditions.
  In this section, we will show that the unique $p$ mean  S asymptotically $\omega$ periodic solution is asymptotically stable in the $p$ mean sense.\\
Recall that 
\begin{definition}
The unique  $p$-mean S asymptotically $\omega$ periodic solution $X^{*}(t)$ of 
 (\ref{eqn: eq3050}) is said to be stable in $p$-mean sense if for any $\epsilon >0$, there exists 
 $\delta > 0$ such that 
 \[    \mathbb{E} || X(t)-X^{*}(t) ||^{p} < \epsilon, \quad \quad t \geq 0,        \]
 whenever $\mathbb{E} || X(0)-X^{*}(0) ||^{p} <  \delta $, where $X(t)$ stands for a solution of
 (\ref{eqn: eq3050}) with initial value $X(0)$.
 \end{definition}
 
 \begin{definition}
 The unique  $p$-mean S asymptotically $\omega$ periodic solution $X^{*}(t)$ is said to be asymptotically
 stable in $p$-mean sense if it is stable in $p$-mean sense  and 
 \[  \lim_{t \rightarrow \infty} \mathbb{E}||  X(t)-X^{*}(t)||^{p} = 0.  \]
 \end{definition}
 
 The following Gronwall inequality is proved to be useful in our asymptotical stability analysis.
 
 \begin{lemma}
 \label{gronwall}
 Let $u(t)$ be  a non negative continuous functions for $t\geq 0$, and $\alpha,\gamma$ be some 
 positive constants. If 
 $$      u(t) \leq \alpha e^{-\beta t} + \gamma \int_{0}^{t} e^{-\beta (t-s)}u(s)ds, \quad t \geq 0, $$
 then $$  u(t) \leq \alpha \exp{ \{ ( -\beta + \gamma ) t \}}.    $$
 \end{lemma}
 
 \begin{theorem}
 \label{stability}
 Suppose that hypothesis {\bf(H.1)}, {\bf(H.2)} and {\bf(H.3)}  are  satisfied and assume that 
 \begin{itemize}
 \item[(i)]
 \begin{equation}
 \label{eq:stabilityp}
3^{p-1} M^{p} \big( L(f) a^{1-p} +  L(g)C_{p}\, a^{\frac{2-p}{2}}  \big) < a 
 \end{equation}
 whenever $p>2$.
 \item[(ii)]
 \begin{equation}
 \label{eq:stability2}
   3M^{2}\big(L(f)a^{-1} + L(g)  \big) < a 
\end{equation}
 whenever $p=2$.
 \end{itemize}
 Then the $p$-mean $S$-asymptotically solution $X_{t}^{*}$ of (\ref{eqn: eq3050}) is asymptotically stable in the $p$-mean sense.
 \end{theorem}
 \begin{rem}
Note that the above conditions (\ref{eq:stabilityp}) respectively (\ref{eq:stability2}) implies
conditions (\ref{eq:existence1}) respectively (\ref{eq:existence2}) of Theorem \ref{existence}.
 \end{rem}

 \begin{proof}
 \begin{align*}
 \mathbb{E} || X(t)-X^{*}(t) ||^{p} & =  \E \Big|\Big| U(t,0)(X(0)-X^{*}(0)) \\
 &\quad \quad + \int_{0}^{t} U(t,s)\big(f(s,X(s))-f(s,X^{*}(s))\big) ds \\
  & \quad \quad + \int_{0}^{t} U(t,s) \big( g(s,X(s))-g(s,X^{*}(s)) \big) dW(s) \Big|\Big|^{p}  
\end{align*}

Assume that $p>2$. Using H\"older inequality  we have
 
\begin{align*}
 \mathbb{E} || X(t)-X^{*}(t) ||^{p} 
 & \leq  3^{p-1}M^{p}e^{-apt}\E || X(0)-X^{*}(0)||^{p}  \\
 & + 3^{p-1} \E \Big(\int_{0}^{t} ||U(t,s)||\, ||f(s,X(s))-f(s,X^{*}(s))|| ds \Big)^{p} \\
 & + 3^{p-1} \E \Big(\int_{0}^{t} ||U(t,s)||\, ||g(s,X(s))-g(s,X^{*}(s))|| dW(s) \Big)^{p} \\
 &\leq    3^{p-1}M^{p}e^{-apt}\E || X(0)-X^{*}(0)||^{p} \\
 & + 3^{p-1}M^{p} \E \Big(\int_{0}^{t} e^{-a(t-s)}\, ||f(s,X(s))-f(s,X^{*}(s))|| ds \Big)^{p}\\
 & +  3^{p-1}M^{p}  \E \Big(\int_{0}^{t} e^{-a(t-s)}\, ||g(s,X(s))-g(s,X^{*}(s))|| dW(s) \Big)^{p} \\
 &= 3^{p-1}M^{p}e^{-apt}\E || X(0)-X^{*}(0)||^{p} + 3^{p-1}M^{p} \times \Big( \\
  & \E \Big(\int_{0}^{t} e^{-\frac{a(p-1)(t-s)}{p}}  e^{-\frac{a(t-s)}{p}}||f(s,X(s))-f(s,X^{*}(s))|| ds \Big)^p \Big) \\
& + 3^{p-1}M^{p} C_{p} \times \Big( \\
& \Big(\int_{0}^{t} e^{-\frac{2a(p-2)(t-s)}{p}}e^{-\frac{4a(t-s)}{p}}\, \E ||g(s,X(s))-g(s,X^{*}(s))||^{2}_{L^{0}_{2}} ds \Big)^{p/2}\Big) \\
  &\leq    3^{p-1}M^{p}e^{-apt}\E || X(0)-X^{*}(0)||^{p}  \\
  & + 3^{p-1}M^{p}L(f)\Big( \int_{0}^{t}e^{-a(t-s)}ds \Big)^{p-1}\times \Big( \\
  &  \int_{0}^{t} e^{-a(t-s)}\E||X(s)-X^{*}(s)||^{p} ds \Big) \\
 &  + 3^{p-1} M^{p} C_{p} \Big(\int_{0}^{t} e^{-2a(t-s)} \Big)^{\frac{p-2}{2}} \times \Big( \\
 & \int_{0}^{t}
 e^{-2a(t-s)}\E \, ||g(s,X(s))-g(s,X^{*}(s))||^{p}_{L_{2}^{0}} ds \Big)
\end{align*} 
so that 
\begin{align*}
 \mathbb{E} || X(t)-X^{*}(t) ||^{p} 
 & \leq 3^{p-1}M^{p}e^{-apt}\E || X(0)-X^{*}(0)||^{p} \\
  & + 3^{p-1}M^{p}L(f) (\frac{1}{a})^{p-1} \int_{0}^{t} e^{-a(t-s)}\E||X(s)-X^{*}(s)||^{p} ds \\
 & + 3^{p-1} M^{p} C_{p} L(g) (\frac{1}{a})^{\frac{p-2}{2}} \int_{0}^{t} e^{-a(t-s)}\E \, ||X(s)-X^{*}(s)||^{p} ds.
\end{align*} 

Using Lemma \ref{gronwall}, we obtain : 

\begin{align*} 
\mathbb{E} || X(t)-X^{*}(t) ||^{p} &  \leq 3^{p-1}M^{p} \times \E || X(0)-X^{*}(0)||^{p} \times \\
& \exp{\left\{\Big(-a + 3^{p-1} M^{p} \big( L(f) a^{1-p} +  L(g) C_{p}\,a^{\frac{2-p}{2}}  \big) \Big)t \right\}}.
\end{align*}

Straightforwardly, we obtain that $X(t)$ converges to $0$ exponentially fast if $$  -a +   3^{p-1} M^{p} \big( L(f) (\frac{1}{a})^{p-1} + L(g) C_{p}\,a^{\frac{2-p}{2}} \big) < 0, $$
which is equivalent to our condition (\ref{eq:stabilityp}). Therefore $X^{*}$ is asymptotically stable in the $p$-mean sense.\\

Assume that $p=2$. We have 
 
\begin{align*}
 \mathbb{E} || X(t)-X^{*}(t) ||^{2} 
 & \leq  3M^{2}e^{-at}\E || X(0)-X^{*}(0)||^{2} \\
 & \quad  +3 \E \Big(\int_{0}^{t} ||U(t,s)||\, ||f(s,X(s))-f(s,X^{*}(s))|| ds \Big)^{2} \\
 & \quad + 3  \E \Big(\int_{0}^{t} ||U(t,s)||\, ||g(s,X(s))-g(s,X^{*}(s))|| dW(s) \Big)^{2} \\
 &\leq   3M^{2}e^{-at}\E || X(0)-X^{*}(0)||^{2}  \\
 &+ 3M^{2} \E \Big(\int_{0}^{t} e^{-a(t-s)}\, ||f(s,X(s))-f(s,X^{*}(s))|| ds \Big)^{2} \\
 &+ 3M^{2}  \E \Big(\int_{0}^{t} e^{-a(t-s)}\, ||g(s,X(s))-g(s,X^{*}(s))|| dW(s) \Big)^{2} 
\end{align*} 
Then using Cauchy-Schwartz inequality   and Lemma 2.3 part (i), we have 
\begin{align*}
\mathbb{E} || X(t)-X^{*}(t)||^{2}
 &\leq  3M^{2}e^{-at}\E || X(0)-X^{*}(0)||^{2} \\
  & + 3M^{2} \int_{0}^{t} e^{-a(t-s)}ds \int_{0}^{t} e^{-a(t-s)} \E ||f(s,X(s))-f(s,X^{*}(s))||^{2} ds \\
 & + 3M^{2}  \int_{0}^{t} e^{-2a(t-s)}\, \E||g(s,X(s))-g(s,X^{*}(s))||^{2}_{L_{2}^{0}} ds, \\
 &\leq   3M^{2}e^{-at}\E || X(0)-X^{*}(0)||^{2} \\
 & + 3M^{2}L(f)a^{-1} \int_{0}^{t} e^{-a(t-s)} \E ||X(s)-X^{*}(s)||^{2} ds \\
 & + 3M^{2}L(g)\int_{0}^{t} e^{-a(t-s)}\, \E|| X(s)-X^{*}(s)||^{2} ds \\
\end{align*} 
Thus 
\begin{align*}
\mathbb{E} || X(t)-X^{*}(t) ||^{2} 
&\leq   3M^{2}e^{-at}\E || X(0)-X^{*}(0)||^{2}\\
& +  \Big( \frac{3M^{2}L(f)}{a} + 3M^{2}L(g)\Big) \int_{0}^{t} e^{-a(t-s)} \E ||X(s)-X^{*}(s)||^{2} ds 
\end{align*} 
By Lemma \ref{gronwall} we have 
$$   \mathbb{E} || X(t)-X^{*}(t) ||^{2} \leq 3M^{2} \E || X(0)-X^{*}(0)||^{2} \exp{\left\{\big( -a + 3M^{2}(\frac{L(f)}{a} + L(g) ) \big)\,t \right\}}          $$

Therefore  $\mathbb{E} || X(t)-X^{*}(t) ||^{2}$ converges to $0$ exponentially fast whenever condition (\ref{eq:stability2}) holds. In particular the unique $S$-asymptotically $\omega$-periodic solution is asymptotically stable in square mean sense.\\

  \end{proof}

 \end{document}